\documentclass{article}

\usepackage[text={5.7in, 7.9in},centering]{geometry}
\usepackage{amsmath,amsthm,xcolor,cite}

\usepackage{bm,amssymb,mathrsfs,url,units,dsfont}
\usepackage{setspace}
\usepackage{bbold}
\usepackage{graphicx,xcolor,cite,mathtools}
\usepackage{amssymb,amsmath,amsthm,mathrsfs,paralist,bm,esint,setspace}
\usepackage{stackengine}
\newcommand\barbelow[1]{\stackunder[1.2pt]{$#1$}{\rule{.8ex}{.075ex}}}
\RequirePackage[colorlinks,citecolor=blue,urlcolor=blue]{hyperref}

\DeclareMathOperator\Cov{Cov}

\DeclareMathOperator\Var{Var}

\renewcommand{\P}{\mathds{P}}
\newcommand{\E}{\mathds{E}}
\newcommand{\R}{\mathds{R}}
\newcommand{\D}{\mathds{D}}
\newcommand{\Z}{\mathds{Z}}

\newcommand\numberthis{\addtocounter{equation}{1}\tag{\theequation}}

\newcommand{\lip}{\text{Lip}}
\newcommand{\Lp}{\text{L}}

\newtheorem{stat}{Statement}[section]
\newtheorem{proposition}[stat]{Proposition}

\newtheorem{theorem}[stat]{Theorem}
\newtheorem{lemma}[stat]{Lemma}
\theoremstyle{definition}

\newtheorem{remark}[stat]{Remark}

\numberwithin{equation}{section}

\begin{document}

\title{Limit theorems for time-dependent averages of nonlinear stochastic heat equations%
	\thanks{Research supported by the NRF (National Research Foundation of Korea) grants 2019R1A5A1028324 and 2020R1A2C4002077.}
}
\author{Kunwoo Kim\\POSTECH  \and Jaeyun Yi \\ POSTECH
}

\date{\today}

\maketitle

\begin{abstract} 
We study limit theorems for  time-dependent averages of the form  $X_t:=\frac{1}{2L(t)}\int_{-L(t)}^{L(t)} u(t, x) \, dx$, as $t\to \infty$, where $L(t)=\exp(\lambda t)$ and $u(t, x)$ is the solution to a stochastic heat equation on $\mathds{R}_+\times \mathds{R}$ driven by space-time white noise with $u_0(x)=1$ for all $x\in \mathds{R}$. We show that for $X_t$
\begin{itemize}
\item[(i)] the weak law of large numbers holds  when $\lambda>\lambda_1$, 
\item[(ii)] the strong law of large numbers holds when $\lambda>\lambda_2$, 
\item[(iii)] the central limit theorem  holds  when $\lambda>\lambda_3$, but fails when $\lambda <\lambda_4\leq \lambda_3$,
\item[(iv)] the quantitative central limit theorem holds when $\lambda>\lambda_5$, 
\end{itemize} 
where $\lambda_i$'s are positive constants depending on the moment Lyapunov exponents of $u(t, x)$. 

\vspace{1cm} 
 
	
\end{abstract}

\section{Introduction and main results}\label{sec1}

We consider the one-dimensional nonlinear stochastic heat equation
\begin{equation}\label{SHE}
     \begin{cases} {\partial \over \partial t}u(t,x) =\frac{1}{2} {\partial^2 \over \partial x^2} u(t,x) + \sigma\left( u(t,x)\right)\dot{W}(t,x), \quad t>0, x\in \R,\\
   u(0,x) = 1, \quad x\in \R,
   \end{cases}
   \end{equation}
   where $\sigma : \R\rightarrow{\R} $ is a Lipschitz continuous function  with the Lipschitz constant $\lip_\sigma>0$  and $\dot{W}$ is a space-time Gaussian white noise. Under these conditions,    it is well known that \eqref{SHE} admits the unique  random field solution (see e.g. \cite{Walsh}) given by
\begin{equation}\label{solution}
    u(t,x) = 1+ \int_{(0,t)\times \R} p_{t-s}(y-x) \sigma \left(u(s,y) \right) W(ds\, dy),\qquad \text{for }t>0, x\in\R,
\end{equation}
where $p_t(x):=(2\pi t)^{-1/2} \exp\left(-x^2/2t\right)$ for $t>0$ and $x\in \R$. 
Since  $u_0(x)=1$ for all $x\in \R$, it is easy to see that $u(t, x)$ is stationary in $x$ (see e.g. \cite{Dalang}).  We further assume that $\sigma $ satisfies    
\[
    \sigma(0)=0 \quad \text{and} \quad  \Lp_\sigma:=\inf_{w\in \R}  \left|\frac{\sigma(w)}{w} \right| >0,
 \]
which guarantees that  $u(t, x)>0$ for all $t\geq 0$ and $x\in \R$ with probability 1 (see e.g. \cite{Mueller, CK}) and  
\begin{equation}\label{eq:cond_sigma}  
 \Lp^2_\sigma \, w^2 \leq \sigma^2(w)\leq \lip_\sigma^2\, w^2 \quad \text{for every $w\in \R$}.
 \end{equation}
Thanks to \eqref{eq:cond_sigma}, we also have the following (see e.g. \cite{FD09, cbms}):
\[ 0< \barbelow{\gamma}(p) \leq \bar{\gamma}(p) < \infty \quad \text{for all $p\geq 1$},\]
where $\bar{\gamma}(p)$ and $\barbelow\gamma(p)$ are 
 the upper and lower $p$-th moment  Lyapunov exponents respectively defined as 
\begin{align}
    \bar{\gamma}(p) &: = \limsup_{t\rightarrow{\infty}} \frac{1}{t} \log \E \left[|u(t,0)|^p \right],\\
     \barbelow{\gamma}(p) &: = \liminf_{t\rightarrow{\infty}} \frac{1}{t} \log \E \left[|u(t,0)|^p \right]. 
\end{align} In addition,  $\bar{\gamma}(1)= \barbelow{\gamma}(1)=0$ since $\E [u(t, x)]=1$ and  $p\mapsto \bar{\gamma}(p)$ is a  convex function, which implies that $\bar{\gamma}$ has a right derivative (see \cite{FD09, cbms}). 

The main objective of this paper is to study limit theorems for  time-dependent averages of the form
\[ \frac{1}{|\Lambda_{L(t)}|}\int_{\Lambda_{L(t)}} u(t, x) dx \qquad \text{as $t\to \infty$},\]
where  $\Lambda_{L(t)}:=\{x\in \R : |x|\leq L(t)\}$ with $L(t):=e^{\lambda t}$ for fixed constant $\lambda>0$. Here are the main theorems: 
\begin{theorem}\label{wlln} (The Weak Law of Large Numbers) If $\lambda>\bar{\gamma}'(1)$, then for every $\epsilon>0$,
\begin{equation}\label{eq:wlln}
    \lim_{t\rightarrow{\infty}} \P\left\{\left| \frac{1}{|\Lambda_{L(t)}|}\int_{\Lambda_{L(t)}} \left( u(t,x)-1  \right) dx    \right| \geq\epsilon \right\} =0,
\end{equation}
where $\bar{\gamma}'$ is defined as  the right derivative of $\bar\gamma$. 
\end{theorem}

\begin{theorem} \label{slln}
(The Strong Law of Large Numbers) 
If $\lambda > \frac{5 \bar{\gamma}(4)}{6}$, then 
\begin{equation}\label{eq:slln}
\frac{1}{|\Lambda_{L(t)}|}\int_{\Lambda_{L(t)}} u(t,x)  dx  \rightarrow{1}, \quad \text{as $t\to \infty$,  a.s..}
\end{equation}
\end{theorem}

\begin{theorem}\label{clt} (The Criteria for the Central Limit Theorem)
Let us define 
\begin{equation}
     F_\lambda(t):=\frac{\int_{\Lambda_{L(t)}} \left( u(t,x)-1  \right) dx  }{\sqrt{\Var{\left(\int_{\Lambda_{L(t)}} u(t,x)dx \right)}}}. 
\end{equation}
Then as $t\rightarrow \infty$:
\begin{itemize}
\item[(i)] $F_\lambda(t) \xrightarrow{\mathcal{L}}N(0,1),$ if $\lambda > \inf_{0<\epsilon<1} \left\{ \frac{2\left(\bar{\gamma}(2+\epsilon)-\barbelow{\gamma}(2)\right)}{\epsilon}-\barbelow{\gamma}(2) \right\} $;
\item[(ii)] $F_\lambda(t) \rightarrow 0$ in probability, if $\lambda < \sup_{0<\epsilon< 1} \left\{ \frac{2\left( \barbelow{\gamma}(2)-  \bar{\gamma}(2-\epsilon)\right)}{\epsilon}-\barbelow{\gamma}(2) \right\} $,
\end{itemize}
where $N(0,1)$ denotes the standard normal distribution.
\end{theorem}

\begin{remark} Since $\bar{\gamma}(p)$ is a convex function for $p\geq 1$ and $\bar{\gamma}(p) \geq \barbelow{\gamma}(p)$ for all $p\geq 1$, we have 
 \[ \inf_{0<\epsilon<1} \left\{ \frac{2(\bar{\gamma}(2+\epsilon)-\barbelow{\gamma}(2))}{\epsilon}-\barbelow{\gamma}(2) \right\} \geq \sup_{0<\epsilon< 1} \left\{ \frac{2( \barbelow{\gamma}(2)-  \bar{\gamma}(2-\epsilon))}{\epsilon}-\barbelow{\gamma}(2) \right\}. \]
\end{remark}

\begin{theorem}\label{clt2} (The Quantitative Central limit Theorem) Suppose that
\begin{equation}
    \lambda >2^{9}\textup{$\lip_\sigma^4$}+ \bar{\gamma}(4) - 2\barbelow{\gamma}(2).
\end{equation}  Let $d_{TV}$ denote the total variation distance and $Z\sim N(0,1)$. Then there exists a positive constant $C=C(\lambda, \sigma)$, which is independent of $t$, such that 
\begin{align*}
    d_{TV} \left(F_\lambda(t),Z \right)\leq Ce^{-C t},
\end{align*}
for all sufficiently large $t$.
\end{theorem}

We first note that the above main theorems  are related to intermittency  in the sense of Carmona and Molchanov \cite{CM}. In order to explain this, we assume $\gamma(p):=\bar\gamma(p)=\barbelow\gamma(p)$ for all $p>0$. Then, intermittency is defined in terms of $\gamma(p)$. More precisely,   a random field  $\{u(t, x); t\geq 0, x\in \R\}$ is called \emph{fully intermittent} if
\begin{equation}\label{eq:intermittency}
 \frac{\gamma(p)}{p} < \frac{\gamma(p+1)}{p+1} \quad \text{for all $p\geq 1$ }. 
 \end{equation}
This full intermittency condition, along with the ergodicity of the random field, implies that the random field develops many different tall peaks over small regions (\emph{intermittent islands}) separated by large areas (\emph{voids}) (see \cite{CM, cbms, CKNP3}; see also \cite{KKX, KKX2} which characterize the geometric structure of the tall peaks).  Here, the condition \eqref{eq:intermittency} tells us the relation between the sizes of the intervals on which Theorems \ref{wlln} and \ref{clt} hold. In other words, if the solution $u(t, x)$ to \eqref{SHE}  is fully intermittent and $\gamma(p):=\bar\gamma(p)=\barbelow\gamma(p)$,  \eqref{eq:intermittency} implies that
\begin{equation}
\begin{aligned}
\lambda_*:= \inf_{0<\epsilon<1} \left\{ \frac{2({\gamma}(2+\epsilon)-{\gamma}(2))}{\epsilon}-{\gamma}(2) \right\} & =  \lim_{\epsilon \downarrow 0} \left\{ \frac{2({\gamma}(2+\epsilon)-{\gamma}(2))}{\epsilon}-{\gamma}(2) \right\}\\
& = 2\gamma'(2)-\gamma (2)\\
& = \lim_{\epsilon \downarrow 0} \left\{ \frac{2({\gamma}(2)-{\gamma}(2-\epsilon))}{\epsilon}-{\gamma}(2) \right\} \\
& =  \sup_{0<\epsilon< 1} \left\{ \frac{2( {\gamma}(2)-  {\gamma}(2-\epsilon))}{\epsilon}-{\gamma}(2) \right\}.
\end{aligned}
\end{equation}
Thus,   the central limit theorem (Theorem \ref{clt}) holds if $\lambda>\lambda_*$ and fails if $\lambda<\lambda_*$. We can also compare Theorem \ref{clt} to Theorem \ref{wlln}. 
 If $\gamma$ is differentiable twice,
\[ \frac{d}{d p} \left(p\gamma'(p) - \gamma(p)\right) = p \gamma''(p) \geq 0,\footnote{Note that  \eqref{eq:intermittency} does not guarantee $\gamma''(p)>0$. The counterexample would be the case  where a random variable has  the $p$-th moment Lyapunov exponent as   $\gamma(p)=p-1$ for $p\geq 1$.} \] since $\gamma$ is convex. Thus, we have 
\[ 2\gamma'(2)-\gamma(2) \geq \gamma'(1).\] 
This says  that the central limit theorem (Theorem \ref{clt}) requires a bigger interval $\Lambda_{L(t)}$ than the weak law of large numbers (Theorem \ref{wlln}) if $\bar{\gamma}(p) = \barbelow{\gamma}(p)$ and $u(t, x)$ is fully intermittent. Indeed, when $\sigma(w)=w$ (in this case,  \eqref{SHE} is called the parabolic Anderson model), Ghosal and Lin \cite{GY} recently showed that $\gamma(p):=\bar{\gamma}(p) = \barbelow{\gamma}(p) = \frac{p(p^2-1)}{24}$ for all $p>0$ (see also \cite{Chen2} for positive integer valued $p$ and \cite{DT} for delta initial data). Therefore, for the parabolic Anderson model, we can find how large the interval should be in order for Theorems \ref{wlln}--\ref{clt2} to hold.  For the general $\sigma$, one may use the moment comparison principle (see  \cite{JKM, CK2}) to get the bounds on $\barbelow\gamma$ and $\bar\gamma$. 

 Theorems \ref{wlln} and \ref{slln} imply that 
  \[ \int_{\Lambda_{L(t)}} u(t,x)\, dx \asymp \E\left[ \int_{\Lambda_{L(t)}} u(t,x)dx \right] = 2 L(t)=2e^{\lambda t},\] and  Lemma \ref{varlower} says  that  
  \[ \Var{\left(\int_{\Lambda_{L(t)}} u(t,x)dx \right)} \asymp L(t) e^{\gamma(2) t}=e^{(\lambda + \gamma(2)) t }.\] 
On the other hand,  if we consider  \eqref{SHE} with $\sigma(w)=1$ for all $w\in \R$ and call the solution $Z(t, x)$, i.e., 
\[ Z(t, x) = 1+ \int_{(0,t)\times \R} p_{t-s}(y-x)\,   W(ds\, dy),\] then  it is easy to see that  $\int_{\Lambda_{L(t)}} Z(t, x) \, dx$ is a Gaussian random variable with mean $2L(t)$ and 
\[ \Var{\left(\int_{\Lambda_{L(t)}} Z(t,x)dx \right)} \asymp L(t).\] Therefore, the central  limit theorems (Theorems \ref{clt} and \ref{clt2}) may provide a quantification that  the tall peaks in the parabolic Anderson model can occur with a higher probability than the tall peaks of the same height for $Z(t, x)$. 
  
Limit theorems for  time-dependent averages of random fields have been mostly studied for random fields on $\Z^d$  (see \cite{AMR, GS, CM2} to only name a few). In particular, Cranston and Molchanov in \cite{CM2} considered the parabolic Anderson model on $\Z^d$ and obtained  the weak law of large numbers  and the central limit theorem that correspond  to Theorems \ref{wlln} and \ref{clt} (i) (Theorem \ref{clt} (ii) did not appear in \cite{CM2}).  They also showed \emph{quenched asymptotics} (i.e., $t^{-1}\log \sum_{x \in \Lambda_{L(t)}} u(t, x) \to \tilde{\gamma}$ a.s. where $\tilde\gamma$ is the almost sure Lyapunov exponent) if $L(t)$ grows slower than exponentially, and  found the upper bound for the quenched asymptotics when $\lambda$ is small (they called it \emph{Transition Range}). Those quenched asymptotics and transition range result were simply obtained from the  known result about  the almost sure Lyapunov exponent (\cite{CMS}) and the additivity of the time-dependent average (i.e., $\sum_{x \in \Lambda_{L(t)}} u(t, x) = u(t, x_1)+\dots + u(t, x_{L(t)})$). 

On the other hand, when the spatial space is $\R^d$,  most of the work has been done recently for quantitative central limit theorems for spatial averages of the form $\int_{[-R, R]^d} g(u(t, x)) \, dx$ as $R\to \infty$ for fixed $t>0$  by beautifully using  the Malliavin calculus and Stein's method where $g: \R \to \R$ is a function such that $g(u)=u$ (see \cite{HNV}),  $g$ is a globally  Lipschitz function (see \cite{CKNP}), or $g$ is a locally Lipschitz function (see \cite{CKNP2}).  The only literature  which mentions about the  time-dependent average that we know is the recent one by Chen et al in \cite{CKNP2} where they show the quantitative central  limit theorem for time-dependent averages  of the form $\int_{[-N, N]^d} u(t_N, x) \, dx$  when $t_N=o(\log N)$ as $N\to \infty$ (see \cite[Corollary 2.7]{CKNP2}). However, a precise box size was not given there.  As far as the strong law of large numbers of the form \eqref{eq:slln}   is concerned, to the best of our knowledge, there have been no results even for the parabolic Anderson model on $\Z^d$ and Theorem \ref{slln} is new. On the other hand, when $L(t)$ is small or $\lambda$ is small, we were not able to obtain the \emph{quenched asymptotics} either  the \emph{transition range} result as in \cite{CM2} since we do not have the additivity of the average, and this could be considered as future work.  Instead, we show that the central limit theorem fails when $\lambda$ is small (see Theorem \ref{clt} (ii)), which was not considered in \cite{CKNP2, CM2}.

We now highlight some main ideas of the proofs of the limit theorems above. For Theorems \ref{wlln} and \ref{clt}, we use  the von Bahr-Esseen inequalities and Lyapunov's criterion as in \cite{CM2}. To use those, one needs to produce independent random variables which are close to the solution to \eqref{SHE}. In \cite{CM2}, Cranston and Molchanov  used the Feynman-Kac formula for the solution to produce independent random variables. However, there is no formal  Feynman-Kac formula (as in  \cite{CM2}) for the solution to \eqref{SHE} (there is a renormalized Feynman-Kac formula with Wick exponential, but it may not be useful here). Thus, we use  the localization argument developed by Conus et al in \cite{CJK} and further quantified  by Khoshnevisan et al in \cite{KKX2}. By the localization argument, we can see that  the solution is \emph{localized} so that whenever $x$ and $y$ are far apart, depending on time variable $t$, $u(t, x)$ and $u(t, y)$ are almost independent (see Lemma \ref{lem:localization}). We also note that along with the localization argument,  we use a quantitative form of the Kolmogorov continuity theorem in an elegant way to show Theorem \ref{slln} (see Lemma \ref{lem:continuity_ave}). Indeed, our methods can be applied to a more general setting such as time-dependent averages of the form $ \frac{1}{|\Lambda_{L(t)}|}\int_{\Lambda_{L(t)}} g(u(t, x) )\, dx $ for some function $g$.  For example, by following the proofs of Theorems \ref{wlln} and \ref{slln}, one can easily show that  Theorems \ref{wlln} and \ref{slln} holds when  $g: \R \to \R$ is a globally Lipschitz function. That is, one may show  that 
\[ \frac{1}{|\Lambda_{L(t)}|}\ \int_{\Lambda(t)} \left\{g(u(t, x)) - \E [g(u(t, x))] \right\} \,dx\]
converges to 0 in probability or almost surely whenever $L(t)$ satisfies the assumptions in Theorems \ref{wlln} and \ref{slln} (we leave this for the interested reader). Lastly, for Theorem \ref{clt2}, we follow the  method of \cite{HNV}, which combines the Malliavin calculus and Stein's method.  

The rest of the paper is organized as follows. In Section \ref{sec2}, we introduce some lemmas and partitions of an interval  used for the proofs of Theorems \ref{wlln}--\ref{clt}. In the following sections (Sections \ref{sec3}--\ref{sec6}),  we provide the proofs of Theorems \ref{wlln}--\ref{clt2} respectively.

\section{Preliminaries}\label{sec2}

In this section, we introduce some  lemmas and partitions of the interval $[-L, L]$ which play a significant role on the proofs of the main theorems. %

\subsection{Some important lemmas}
The following lemma  is the von Bahr-Esseen inequality which will be used in the proof of Theorem \ref{wlln}. 
\begin{lemma}[von Bahr-Esseen inequality \cite{BE}] \label{bahresseen}Let $\{ X_n \} _{n\in \mathbb{Z}^+}$ be a sequence of independent mean zero random variables, and $1\leq r \leq 2$. Then 
\begin{equation}
    \E \left[\lvert\sum_{i=1}^n X_i\rvert^r  \right] \leq 2 \E \left[\sum_{i=1}^n |X_i|^r \right],
\end{equation}
for all integers $n\geq 1$.
\end{lemma}
We note that if we just use Jensen's inequality, we get $n^{r-1}$ which is greater than 2 for large $n$ in the von Bahr-Esseen inequality. In other words, the von Bahr-Esseen inequality gives a better bound than Jensen's inequality and it plays an essential  role 
in the proof of Theorem \ref{wlln}. The following lemma is the (generalized) von Barh-Esseen inequality when $r \geq 2$ and this inequality also provides a better bound than Jensen's inequality. We will use this inequality for the proof of Theorem \ref{slln}. 

\begin{lemma}[Generalized von Bahr-Esseen inequality \cite{BE2}]\label{Bahresseen2} Let $\{X_n\}_{n \in \mathbb{Z}^+}$ be a sequance of independent mean zero random variables and $r\geq 2$. Then, there exists a finite constant $c_r>0$ which is independent of $n$ such that 
\begin{equation}
    \E \left[\lvert\sum_{i=1}^n X_i\rvert^r  \right] \leq c_r n^{\frac{r}{2}-1} \E \left[\sum_{i=1}^n |X_i|^r \right],
\end{equation}
for all integers $n\geq 1$.

\end{lemma}

In order to apply the von Bahr-Esseen inequalities,  random variables must  be  independent. We  also need independence when we use Lyapunov's criterion  to show the central limit theorem (Theorem \ref{clt}).  Thus, we introduce a localization argument that was introduced by Conus et al in \cite{CJK}. This localization argument produces independent random variables which are close to the solution $u(t, x)$ of \eqref{SHE}. 

Let us define intervals $I(x,t;c):= \left[x-\sqrt{ct},  x+\sqrt{ct} \right]$ for some fixed constant $c>0$.
Let $u^{(c,0)}(t,x):=1$ for all $t\geq 0 $ and $x\in \R$, and then define 
\begin{equation}
    u^{(c,n+1)}(t,x) = 1+ \int_{(0,t)\times I(x,t:c)} p_{t-s}(y-x) \sigma \left(u^{(c,n)}(s,y) \right) W(dsdy),
    \end{equation}
    iteratively for all $n\geq 0 $.
It is easy to see that $u^{(c, n)}(t, x)$ and $u^{(c, n)}(t, y)$ have the same distribution and are independent whenever $|x-y|\geq 2n\sqrt{ct}$. The following lemma which is basically Theorem 3.9 of \cite{KKX2} says that for appropriate $c$ and $n$, we can have independent random variables which approximate $u(t, x)$ (more precisely, see (3.23) and (3.24) in \cite{KKX2} with $\mu=1$). 

\begin{lemma}[Localization of the solution] \label{lem:localization}
For all $t\geq 1$ and $k\geq 2$,  there exists a constant $A>0$   such that $u^{(k)}(t,x):= u^{(Ak^2t,\lceil Ak^2t \rceil)}(t,x)$ satisfies the following: 
There exists a constant $C$ independent of $k,t$ such that 
\begin{equation}\label{eq:error:u-uk}
   \sup_{x\in\R} \E\left[ |u(t,x)-u^{(k)}(t,x)|^k  \right] \leq 2C^k e^{- k^3 t}.
\end{equation}
Moreover, there exists a finite constant $c_0>0$ independent of $k$ and $t$ such that whenever nonrandom points $x_1,...,x_m \in \R$  satisfy 
\begin{equation}\label{localization1}
    \min_{1\leq i\neq j \leq m } |x_i-x_j| > c_0t^2 k^3,
\end{equation}
$u^{(k)} (t, x_1), \dots, u^{(k)} (t, x_m)$ are i.i.d. random variables.
\end{lemma}
We note that \eqref{eq:error:u-uk} easily implies that for all $t\geq 1$ and $k\geq 2$ 
\begin{equation}\label{moments}
e^{\left( \barbelow\gamma (k) + o(1)\right) t } \leq \E \left[ \left| u^{(k)}(t,x) \right|^k \right] \leq e^{\left( \bar\gamma (k) + o(1)\right) t }   
\end{equation}
as $t\to \infty$.  In addition, by following exactly the same proof of Lemma \ref{lem:localization}, we can also see that $u^{(k)} (t, x_1)-u^{(k)}(s,x_1), \dots, u^{(k)} (t, x_m)-u^{(k)}(s,x_m)$ are independent random variables whenever $ \min_{1\leq i\neq j \leq m } |x_i-x_j| > c_0 (t\vee s)^2 k^3$. 
\subsection{Partitions of the interval $\Lambda_L:=[-L, L]$}\label{sec:partitions}
 Let us introduce some partitions of  the interval  $\Lambda_L:=[-L, L]$, which was introduced by Ben Arous et al in \cite{AMR}, so that  we can use the independence of $u^{(k)}(t, x)$ and $u^{(k)}(t, y)$ whenever $x$ and $y$ are far apart.  Given $L>0$, let $0<L'<L$ and $q:=\lfloor\frac{2L}{L'}\rfloor$. Define $L_j' = L' +\alpha_j1_{\{1,2,...,q\}}(j)$ for some $\alpha_j\in[0,1]$ so that $2L= \sum_{j=1}^q L_j'$. Note that $L'\leq L'_j \leq L'+1$ for all $1\leq j\leq q$. We now define a partition $\{\Lambda_i\}_{i=1}^q$  of $\Lambda_L$ as 
\begin{equation}\label{partition1}
    \Lambda_i=\left[-L+ \sum_{j=1}^{i-1}L_j', -L+ \sum_{j=1}^{i}L_j' \right]\quad \text{for $i=1, \dots, q$},
\end{equation}
 with the convention that $\sum_{j=1}^0L_j'=0.$ Let  the index set $\mathcal{I}=\{1,2,...,q\}$ and  let us partition $\mathcal{I}$ into $\mathcal{I}_{1} \cup \mathcal{I}_{2},$ where $\mathcal{I}_{1}=\{i\in \mathcal{I}: i \text{ is even} \}$ and $\mathcal{I}_{2}=\{i\in \mathcal{I}: i \text{ is odd} \}$. Then, for $i,j\in\mathcal{I}_{\ell}$ with $i\neq j$, we have $d(\Lambda_i, \Lambda_j)\geq L'$ for $\ell=1,2$.  These partitions are heavily used in the proofs of Theorems \ref{wlln} and \ref{slln}. 
 
 We now introduce another decomposition of $\Lambda_L$ for the proof of Theorem \ref{clt}. Define subintervals $\Lambda'_i \subset \Lambda_i$ as 
\begin{equation}\label{partition2}
    \Lambda'_i= \left[-L+ \sum_{j=1}^{i-1}L_j'+ \lceil c_0t^2k^3\rceil, -L+ \sum_{j=1}^{i}L_j'-\lceil c_0t^2k^3 \rceil\right]\quad \text{for $i=1, \dots, q$ } 
\end{equation} 
and set the strip set $S_L := \Lambda_L \setminus \bigcup_{i\in \mathcal{I}} \Lambda'_i$ so that $\Lambda_L = S_L \cup \bigcup_{i\in \mathcal{I}}\Lambda'_i $. We also   partition $S_L$ as $\{S_{L, i}\, : i=0, \dots, q\}$ where $S_{L, 0}:=\left[-L, -L+\lceil c_0t^2k^3\rceil\right]$, $S_{L, q}:=\left[L-\lceil c_0t^2k^3\rceil , L\right]$  and 
\begin{equation}\label{eq:partition_SL}
  S_{L, i}:=\left[-L+ \sum_{j=1}^{i}L_j'-\lceil c_0t^2k^3 \rceil, -L+ \sum_{j=1}^{i}L_j' + \lceil c_0t^2k^3 \rceil\right] \quad \text{for $i=1, \dots, q-1$}. 
  \end{equation}
Here, $c_0$ is the constant that appeared in Lemma \ref{lem:localization}. Note that $u^{(k)} (t, x)$ and $u^{(k)}(t, y)$ are independent as long as $x\in \Lambda_i'$ and $y\in \Lambda_j'$  or  $x\in S_{L,i}$ and $y\in S_{L,j}$  for $i\neq j$.

For Theorems \ref{wlln}, \ref{clt} and \ref{clt2}, we basically take $L:=L(t):=e^{\lambda t}$ and $L':=L'(t):=e^{\lambda' t}$ for some constants $\lambda> \lambda' >0$ (unless they are specified otherwise). On the other hand, we specify $L$ and $L'$ for Theorem \ref{slln} in Section \ref{sec4}.

\section{Proof of Theorem \ref{wlln}} \label{sec3}
In this section, we give a proof of Theorem \ref{wlln} (the weak law of large numbers).  We first show that the spatial average of $u^{(k)}$  over the interval $\Lambda_L:=[-L,L]$ is close to  the spatial average of $u$ over the same interval $\Lambda_L$ for all large $t$, no matter what the interval length ($2L:=|\Lambda_L|$) can be. 

\begin{lemma}\label{lem:approxwlln}
Let $\Lambda_L:=[-L,L]$. For every $\epsilon>0$ and $k\geq 2$, we have 
\begin{equation}
\lim_{t\to \infty}\, \sup_{L\geq 1}  \P \left\{\left| \frac{1}{|\Lambda_L|}\int_{\Lambda_L} \left( u(t,x)-1  \right) dx -\frac{1}{|\Lambda_L|}\int_{\Lambda_L} \left( u^{(k)}(t,x)-1  \right) dx       \right| \geq\epsilon \right\} =0
\end{equation}
\end{lemma}
\begin{proof}
Jensen's inequality and Lemma \ref{lem:localization} implies that for all $t\geq 1$
\begin{align*}
    \E\left[  \left| \frac{1}{|\Lambda_L|}\int_{\Lambda_L} \left( u(t,x)-u^{(k)}(t,x)  \right) dx  \right|^2 \right]&\leq  \frac{1}{|\Lambda_L|} \int_{\Lambda_L} \E\left[\left| u(t,x)-u^{(k)}(t,x)  \right|^2 \right]dx\\
    &\leq  \sup_{x\in\R} \E\left[ |u(t,x)-u^{(k)}(t,x)|^2  \right] \\
    &\leq \sup_{x\in\R} \E\left[ |u(t,x)-u^{(k)}(t,x)|^k  \right]^{\frac{2}{k}} \\
    &\leq 2C^2 e^{-2 k^2 t},
\end{align*}
where the constant $C$ is  given in Lemma \ref{lem:localization}, which is  independent of $L$.  We now  use Chebyshev's inequality to complete the proof.
\end{proof}

The following lemma gives us an estimation of the moments of $\int_{\Lambda_L} (u^{(k)}(t,x)-1) dx$ in which  Lemma \ref{bahresseen} plays a crucial role. For simplicity, we consider in this section  $\Lambda_L:=[-L, L]$,  $L:=L(t):=e^{\lambda t}$ and $L':=e^{\lambda' t}$ for some fixed constants $\lambda>\lambda'>0$ (see also Section \ref{sec:partitions} for the partitions of $\Lambda_L$).

\begin{lemma}\label{varestimate} For all $1\leq r \leq 2$ and $k\geq 2$, we have 
\begin{equation}\label{centeredmoment}
    \E \left[\left|\int_{\Lambda_L} (u^{(k)}(t,x)-1) dx \right|^r \right] \leq C_r  \exp \left\{ \left(\lambda + \lambda'(r-1)+\bar{\gamma}(r)+o(1) \right) t\right\}
\end{equation}
as $t\rightarrow{\infty}$, where $C_r$ is a positive constant only depending $r$.
\end{lemma}

\begin{proof}
From Lemma \ref{lem:localization}, we know that $u^{(k)}(t,x)-1$ and $u^{(k)}(t,y)-1$ are independent whenever $x\in \Lambda_i$ and $y\in \Lambda_j$ with  $i, j \in \mathcal{I}_l$ and  $i\neq j$ for each $l=1,2$. Indeed, for all large enough $t$, $|\Lambda_i| \geq 2e^{\lambda't} \gg c_0t^2k^3 $ for all $k\geq 2$, which implies the independence. Once we  decompose $\Lambda_L$ as 
\[
\Lambda_L=\bigcup_{\ell=1,2} \bigcup_{i\in \mathcal{I}_\ell}\Lambda_i, 
\]
by using Jensen's inequality, Lemma \ref{bahresseen} and Jensen's inequality, respectively, in lines two, three and four below, we see that 

\begin{align*}
    \E \left[\left|\int_{\Lambda_L} (u^{(k)}(t,x)-1)dx \right|^r \right] &= \E \left[\left|\sum_{\ell=1,2} \sum_{i\in\mathcal{I}_\ell}    \int_{\Lambda_i} (u^{(k)}(t,x)-1)dx \right|^r \right] \\
    &\leq 2^{r-1} \sum_{\ell=1,2} \E \left[\left| \sum_{i\in\mathcal{I}_\ell}    \int_{\Lambda_i} (u^{(k)}(t,x)-1)dx \right|^r \right]\\
    &\leq 2^{r}\sum_{\ell=1,2}\sum_{i\in\mathcal{I}_\ell}  \E \left[\left|    \int_{\Lambda_i} (u^{(k)}(t,x)-1) dx \right|^r \right]\\
    &\leq 2^{r}\sum_{\ell=1,2}\sum_{i\in\mathcal{I}_\ell} |\Lambda_i|^{r-1}  \int_{\Lambda_i} \E \left[\left|    u^{(k)}(t,x)-1 \right|^r \right]dx.
\end{align*}
It is clear that from \eqref{moments}, for all $x\in\R$,
\begin{align*}
    \E \left[\left|    u^{(k)}(t,x)-1 \right|^r \right] \leq 2^r \E \left[\left|    u^{(k)}(t,x) \right|^r \right] \leq 2^r e^{(\bar{\gamma}(r)+o(1))t},
\end{align*}
as $t\rightarrow{ \infty}$. Therefore, we conclude that 
\begin{align*}
    \E \left[\left|\int_{\Lambda_L} (u^{(k)}(t,x)-1)dx \right|^r \right] &\leq  2^{r}\sum_{k=1,2}\sum_{i\in\mathcal{I}_k} |\Lambda_i|^{r-1}  \int_{\Lambda_i} 2^r\E \left[\left|    u^{(k)}(t,x) \right|^r \right]dx\\
    & \leq C_r(L'+1)^{r-1} (2L)e^{(\bar{\gamma}(r)+o(1))t}
\end{align*}
since $|\Lambda_i| \leq L'+1$ for all $i$.
\end{proof}

We are now ready to prove Theorem  \ref{wlln}. 
\begin{proof}[Proof of Theorem \ref{wlln}] 
We first note that thanks to Lemma \ref{lem:approxwlln}, it suffices to show that the weak law of large numbers (i.e. \eqref{eq:wlln})  holds for $u^{(k)}(t,x)$. Since $|\Lambda_L|=2L(t)=2e^{\lambda t}$, Lemma \ref{varestimate} implies that,  for $0<\delta<1$, as $t\rightarrow{\infty}$, we have 
\begin{align} \label{vonbahr}
      \E \left[\left| \frac{1}{|\Lambda_L|} \int_{\Lambda_L} (u^{(k)}(t,x)-1)dx \right|^{1+\delta} \right]
      &\leq C_{\delta} \exp \left\{\delta\left(\lambda'-\lambda+ \frac{\bar{\gamma}(1+\delta)}{\delta} +o(1)\right)t    \right\}, 
\end{align}
where $C_{\delta}$ is a positive constant only depending on $\delta$. Since $\lambda> \bar{\gamma}'(1)$ and $\bar{\gamma}(1)=0$, we can take $\lambda'$ and $\delta$ sufficiently small so that 
$$
\lambda'-\lambda + \frac{\bar{\gamma}(1+\delta)-\bar{\gamma}(1)}{\delta}<-c.
$$
for some positive constant $c>0$. Then, by Chebyshev's inequality, taking $t\rightarrow{\infty}$ leads to the desired conclusion. 
\end{proof}

\section{Proof of Theorem \ref{slln}} \label{sec4}
In this section, we provide a proof of Theorem \ref{slln} (the strong law of large numbers). The following proposition tells us the continuity of the solution to \eqref{SHE}, which helps us  control the fluctuation of $\int_{\Lambda_L} (u(t,x)-1)dx$ in time. In this section, we write $L(t)$ and $L(s)$ instead of $L$ where $L(t)=e^{\lambda t}$ and $L(s)=e^{\lambda s}$.
\begin{lemma}\label{contiprop}
 There exists a finite positive constant $D$ which is independent of $n$ such that for all $k\geq 2$ 
\begin{equation}
 \sup_{n\leq t\neq s <n+1}    \sup_{x\in \R}\E\left[ |u(t,x)-u(s,x)|^k \right] \leq D^k e^{\left(\bar{\gamma}(k) +o(1)\right)n}|t-s|^{\frac{k}{4}},
\end{equation}
as $n\to\infty$. 
\end{lemma}
The proof is very similar to that of Lemma 5.4 of \cite{cbms}. To the best of our knowledge, however, the exact coefficient $e^{\bar{\gamma}(k)n}$ has not been exactly identified. Thus we briefly present the proof of Lemma \ref{contiprop} for the sake of completeness.  

\begin{proof}
From the proof of Lemma 5.4 in \cite{cbms} with $\Phi_t(x) = \sigma(u(t,x))$, for all $x\in \R$ and $0<s<t$, we have 
\begin{align*}
 \E\left[ |u(t,x)-u(s,x)|^k \right]^{\frac{2}{k}}&\leq 2 (J_1 +J_2),
\end{align*}
where 
\begin{align*}\label{sigmamoments}
    J_1&:= \text{const}\cdot k \int_0^t\int_{\R} [p_{t-r}(y-x)-p_{s-r}(y-x)]^2 \E\left[\sigma^k(u(r,y))     \right]^{\frac{2}{k}}dyds,\\
    J_2&:=  \text{const}\cdot k \int_s^t\int_{\R} [p_{t-r}(y-x)]^2 \E\left[\sigma^k(u(r,y))     \right]^{\frac{2}{k}}dyds.
\end{align*}
Here,  \eqref{eq:cond_sigma} implies that  as $t\rightarrow{\infty}$,
\begin{equation}
    \sup_{0\leq r \leq t}\sup_{y\in\R} \E\left[\sigma^k(u(r,y))     \right] \leq e^{\left(\bar{\gamma}(k)+o(1)\right)t}.
\end{equation}
In addition, we can also have  that (see \cite[Section 3.3.1]{cbms}) 
\begin{align}
    \int_0^t&\int_{\R} [p_{t-r}(y-x)-p_{s-r}(y-x)]^2 dyds + \int_s^t\int_{\R} [p_{t-r}(y-x)]^2 dyds \leq C  |t-s|^\frac{1}{2},
\end{align}where the constant $C$ is independent of $t,s$ and $x$. 
Putting all the things together, we get \eqref{contiprop}.
\end{proof}

We set $L(t):=e^{\lambda t}$, $L'(t):=e^{\lambda' t}$ and $\Lambda_{L(t)}:=[-L(t), L(t)]$ for fixed constants $\lambda > \lambda'>0$.  We now control the fluctuation of $\int_{\Lambda_{L(t)}} (u(t,x)-1)dx$ in time. For simplicity, we define 
 \[ X(t):= \int_{\Lambda_{L(t)}} (u(t,x)-1)\ dx.\]

\begin{lemma}\label{lem:continuity_ave}
For any $p>4$, there exists a constant $C>0$ which only depends on $p$ and $\lambda$  such that as $n\to \infty$
\begin{equation}\label{claimschain}
  \E\left[ \sup_{t\neq s \in [n,n+1)} |X(t)-X(s)|^p\right]  \leq C \exp \left\{\left( \frac{\lambda p}{2} + \frac{\lambda' p }{2} + \bar{\gamma}(p)+o(1) \right)n\right\},
\end{equation} provided that 
\[ \frac{2\lambda}{5}<\lambda'<\lambda.\]  
\end{lemma}
\begin{proof}
By a quantitative form of the Kolmogorov continuity theorem (see e.g.  \cite[Theorem C.6]{cbms}), it is enough to show that for some $p_2 \in (1,\frac{p}{4})$, there exists a constant $C_{\lambda,p}$ only depending on $\lambda$ and $p$ such that 
\begin{equation}\label{claimschain2}
     \sup_{t\neq s \in [n,n+1)} \E \left[\frac{|X(t)-X(s)|^p}{|t-s|^{p_2}} \right] \leq C_{\lambda,p} \exp \left\{\left( \frac{\lambda p}{2} + \frac{\lambda' p }{2} + \bar{\gamma}(p)+o(1) \right)n\right\},
\end{equation}
as $n\to \infty$. We now fix large $n$ and $t\neq s\in [n,n+1)$. Without loss of generality, we assume that $t>s$.
We  estimate the LHS of \eqref{claimschain2} as follows: 
\begin{equation}
  \E [|X(t)-X(s)|^p]  \leq 4^p ( A_1 + A_2 + A_3 + A_4), 
\end{equation}
where 
\begin{align*}
    A_1:&= \E\left[ \left|\int_{\Lambda_{L(t)}} ( u(t,x)-u^{(k)}(t,x)) dx \right|^p    \right],\\
    A_2:&= \E\left[ \left|\int_{\Lambda_{L(s)}} (u(s,x)-u^{(k)}(s,x)) dx \right|^p    \right],\\
    A_3:&= \E\left[ \left|\int_{\Lambda_{L(s)}} (u^{(k)}(t,x)-u^{(k)}(s,x)) dx \right|^p    \right],\\
    A_4:&= \E\left[ \left|\int_{\Lambda_{L(t)}\setminus\Lambda_{L(s)}} (u^{(k)}(t,x) -1) dx \right|^p    \right],
\end{align*}
where $k$ will be specified later (see \eqref{eq:k} below).
 We will denote by $C$ a constant depending only on $p$ and $\lambda$ and can be changed line by line. By Jensen's inequality and H$\ddot{\text{o}}$lder's inequality, we have 
\begin{align*}
    A_1 \leq C e^{\lambda p t}   \sup_{x\in \R} \E\left[ \left|u(t,x)-u^{(k)}(t,x) \right|^p\right] \leq C e^{\lambda p t}  \sup_{x\in \R} \E\left[ \left|u(t,x)-u^{(k)}(t,x) \right|^k\right]^{\frac{p}{k}},
\end{align*}
for all $k\geq p$. If we  choose $k$ as 
\begin{equation}\label{eq:k}
    k:= k(n,|t-s|): =p\vee  \left( \frac{\lambda-\lambda' }{2} - \frac{\bar\gamma (p)}{p}- \frac{p_2}{n p}\ln|t-s| \right)^{\frac{1}{2}}
\end{equation}
then Lemma \ref{lem:localization} implies that  for all $t\neq s \in [n, n+1)$, 
\begin{equation} \label{eq:u-uk}
\sup_{x\in \R} \E\left[ \left|u(t,x)-u^{(k)}(t,x) \right|^p\right]  \leq  C |t-s|^{p_2} e^{\left(\bar{\gamma}(p)- \frac{\lambda p}{2}+\frac{\lambda 'p}{2}  \right) n},
\end{equation}
which in turn implies that 
\[    A_1 \leq C|t-s|^{p_2}\exp \left\{\left( \frac{\lambda p}{2} + \frac{\lambda' p }{2} + \bar{\gamma}(p) \right)n\right\}. \]
Similarly, we can estimate $A_2$ so that we have 
\begin{equation}
    A_2 \leq  C|t-s|^{p_2}\exp \left\{\left( \frac{\lambda p}{2} + \frac{\lambda' p }{2} + \bar{\gamma}(p) \right)n\right\}.
\end{equation}

Now we estimate $A_3$. Let $\beta$ be a fixed positive constant that will be specified later. We first consider the case where $t\neq s \in [n, n+1)$ with $|t-s| \geq e^{-\beta n}$. In this case, it is easy to see that for all large $n$
\[ e^{\lambda ' n} \gg c_0 k^3 t^2\]
where $c_0$ was the constant that appeared in Lemma \ref{lem:localization} and $k:=k(n, |t-s|)$ was  defined in \eqref{eq:k}. Therefore, we can use the generalized von Bahr-Esseen inequality (Lemma \ref{Bahresseen2}) as in the proof of Lemma \ref{varestimate}. With the same notations (with $s$ instead of $t$) as in  the proof of Lemma \ref{varestimate}, we get that 
\begin{equation}\label{eq:A3}
\begin{aligned}
   A_3 &=  \E \left[\left|\sum_{\ell=1,2} \sum_{i\in\mathcal{I}_\ell}    \int_{\Lambda_i} (u^{(k)}(t,x)-u^{(k)}(s,x))dx \right|^p \right]\\
   &\leq 2^{p-1}  \sum_{\ell=1,2} \E \left[\left| \sum_{i\in\mathcal{I}_\ell}    \int_{\Lambda_i}(u^{(k)}(t,x)-u^{(k)}(s,x)) dx \right|^p \right]\\
   &\leq C \left(\frac{L(s)}{L'(s)}\right)^{\frac{p}{2}-1} \sum_{\ell=1,2}\sum_{i\in\mathcal{I}_\ell} \E \left[\left|    \int_{\Lambda_i} (u^{(k)}(t,x)-u^{(k)}(s,x))dx\right|^p \right]\\
   &\leq C \left(\frac{L(s)}{L'(s)}\right)^{\frac{p}{2}-1} \sum_{\ell=1,2}\sum_{i\in\mathcal{I}_\ell} (L'(s))^{p-1} \int_{\Lambda_i} \E \left[\left|  u^{(k)}(t,x)-u^{(k)}(s,x) \right|^p \right]dx\\
   &\leq C e^{\left( \frac{\lambda p}{2}+\frac{\lambda' p}{2}  \right)n} \sup_{x\in \R} \E \left[\left|  u^{(k)}(t,x)-u^{(k)}(s,x) \right|^p \right].\\
\end{aligned}
\end{equation}
Note that for all $x\in \R$, 
\begin{align}\label{tri}
    \E &\left[\left|  u^{(k)}(t,x)-u^{(k)}(s,x) \right|^p \right] \\ &\leq C \left\{ \E \left[\left|  u^{(k)}(t,x)-u(t,x) \right|^p \right] + \E \left[\left|  u^{(k)}(s,x)-u(s,x) \right|^p \right]+\E \left[ \left|  u(t,x)-u(s,x) \right|^p \right]\right\}. 
\end{align}
Here, the first two terms on the right hand side above can be bounded as in \eqref{eq:u-uk}. On the other hand, Lemma \ref{contiprop} implies that the last term is bounded by $D^p e^{(\bar{\gamma}(p) + o(1))n}|t-s|^{\frac{p}{4}}$. Since $p_2 < p/4$, as long as $|t-s|\geq e^{-\beta n}$, we get 
\begin{equation}\label{eq:A3-2} A_3 \leq  |t-s|^{p_2} \exp \left\{\left( \frac{\lambda p}{2} + \frac{\lambda' p }{2} + \bar{\gamma}(p)+o(1) \right)n\right\}.
\end{equation}
We now consider $A_3$ when $0< |t-s| < e^{-\beta n}$. First of all,  by Jensen's inequality, 
\begin{align*}
   A_3 \leq Ce^{\lambda p n} \sup_{x\in \R}\E \left[\left|  u^{(k)}(t,x)-u^{(k)}(s,x)\right|^p \right].
\end{align*}
As in  \eqref{tri},   \eqref{eq:u-uk} makes us only bound  the term  $e^{\lambda p n}\sup_{x\in \R}\E \left[ \left|  u(t,x)-u(s,x) \right|^p \right]$.  Once again, by Lemma \ref{contiprop}, we have
\begin{equation}
     \sup_{x\in \R}\E \left[ \left|  u(t,x)-u(s,x) \right|^p \right] \leq   D^p e^{\left(\bar{\gamma}(p) + o(1)\right)n}|t-s|^{\frac{p}{4}}.
\end{equation} 
Since $|t-s|<e^{-\beta n}$, we choose  $\beta > 2(\lambda -\lambda')p / (p-4p_2)$  to get  that
\begin{equation}
 D^p e^{\left(\bar{\gamma}(p)+\lambda p \right)n}|t-s|^{\frac{p}{4}} \leq  e^{\left( \frac{\lambda p}{2} + \frac{\lambda' p }{2} + \bar{\gamma}(p) \right)n}|t-s|^{p_2}, 
\end{equation}
and this implies \eqref{eq:A3-2} when $|t-s|<e^{-\beta n}$.

Finally, we estimate $A_4$. Note that for  $x\in [0,1]$, there exist constants $\lambda_1$ and $\lambda_2$ which only depend on $\lambda$ such that  $\lambda_1 x \leq 1-e^{-\lambda x} \leq \lambda_2 x$. Thus, we have 
\[ \lambda_1(t-s) e^{\lambda t} \leq L(t)-L(s)  \leq  \lambda_2 (t-s) e^{\lambda t}.\] Let $\eta$ be a fixed positive constant that will be specified later. For $0< |t-s| \leq e^{-\eta n}$, by Jensen's inequality, we have 
\begin{align*}
    A_4 &\leq \left(     e^{\lambda t} \lambda_2 (t-s) \right)^p \sup_{x\in\R} \E\left[\left|u^{(k)}(t,x)-1 \right|^p\right]\\
    &\leq C e^{\lambda pn} |t-s|^p e^{\left(\bar{\gamma}(p)+o(1)\right)n}\\
    &\leq\exp \left\{ \left(\frac{\lambda p}{2} + \frac{\lambda' p}{2} +\bar{\gamma}(p) +o(1)     \right)n \right\}|t-s|^{p_2},
\end{align*}
as $n\rightarrow{\infty}$, provided that
\begin{equation}\label{parameter2}
   \frac{ \lambda p}{2}-\frac{\lambda' p }{2} \leq \eta (p-p_2).
\end{equation}
We now consider $A_4$ when $e^{-\eta n} \leq t-s\leq 1$. As we bound  $A_3$ when $|t-s|$ is not too small,  we  again use the generalized von Bahr-Esseen inequality (Lemma \ref{Bahresseen2}). Here, we divide the interval of length $L(t)-L(s)$ into the subintervals of length $e^{\lambda'n}(t-s)$. We also assume $\eta < \lambda'$ so that for all large $n$
\begin{align*}
    e^{\lambda' n}(t-s)\gg c_0 t^2 k^3, 
\end{align*}
where $k$ is in \eqref{eq:k}. This guarantees the independence between  $u^{(k)}(t,x)$ and $u^{(k)}(t,y)$ when $x\in \Lambda_i$ and $y\in \Lambda_j$ with  $i, j \in \mathcal{I}_l$ and  $i\neq j$ for each $l=1,2$. As for $A_3$ (see \eqref{eq:A3}, by considering $L$ and $L'$ as $e^{\lambda t}-e^{\lambda s}$ and $e^{\lambda' t}(t-s)$, we have 
\begin{align*}
    A_4 &\leq C  e^{\left( \frac{\lambda p}{2}+\frac{\lambda' p}{2}  \right)n} |t-s|^p \sup_{x\in\R}\E\left[ |u^{(k)}(t,x)-1|^p  \right]\\
    &\leq C|t-s|^p \exp\left\{\left(\frac{\lambda p}{2}+\frac{\lambda' p}{2}+\bar{\gamma}(p)+o(1) \right)n \right\}\\
    &\leq C|t-s|^{p_2} \exp\left\{\left(\frac{\lambda p}{2}+\frac{\lambda' p}{2}+\bar{\gamma}(p)+o(1) \right)n \right\}.
\end{align*}
Therefore, the only remaining thing is to verify \eqref{parameter2} for some $\eta \in (0, \lambda')$ and $p_2\in (1, p/4)$. Using the assumption that $\lambda' \in (2\lambda/5, \lambda)$, this can be easily done by choosing $p_2=p/4 - \epsilon_1$ and $\eta=\lambda'-\epsilon_2$ for some small $\epsilon_1, \epsilon_2>0$. 
\end{proof}

We now provide the proof of Theorem \ref{slln}.

\begin{proof}[Proof of Theorem \ref{slln}]
We first consider $\E|X(n)|^p$ for $p>4$ where  $X(t):= \int_{\Lambda_{L(t)}} (u(t,x)-1)\ dx$. Similar to Lemma \ref{varestimate}, we use Jensen's inequality, the generalized von Bahr-Esseen inequality (Lemma \ref{Bahresseen2}) and Lemma \ref{lem:localization} to get that, as $n\rightarrow{\infty}$,
\begin{align*}
    \E\left[ |X(n)|^p  \right]& = \E \left[ \left| \int_{\Lambda_{L(n)}} (u(n,x)-1)  dx \right|^p\right]\\ 
    &\leq C \left\{ \E\left[ \left| \int_{\Lambda_{L(n)}} (u(n,x)-u^{(k)}(n,x))dx \right|^p\right]
    + \E\left[ \left| \int_{\Lambda_{L(n)}} (u^{(k)}(n,x)-1)dx \right|^p\right] \right\}\\
    &\leq C\left\{ e^{\lambda p n}\sup_{x\in \R} \E \left[\left|  u(n,x)-u^{(k)}(n,x) \right|^p \right]+ e^{\left(\frac{\lambda p}{2} +\frac{\lambda' p }{2}\right)n}\sup_{x\in \R} \E \left[\left|  u^{(k)}(n,x)-1 \right|^p \right]\right\}\\
     &\leq C \left\{ \exp\left( (\lambda p  -  k^2 p )n \right) + \exp \left( \left( \frac{\lambda p}{2} + \frac{\lambda' p}{2} + \bar{\gamma}(p) +o(1)\right) n \right)\right\}\\
     & \leq C \exp \left( \left( \frac{\lambda p}{2} + \frac{\lambda' p}{2} + \bar{\gamma}(p) +o(1)\right) n \right),
\end{align*}  where $C>0$ is a constant only depending $p$ and $\lambda$, which can be changed line by line.  
Of course, we take  $k > p \vee \sqrt{\lambda}$ in the last inequality above. Combining this  with  Lemma \ref{lem:continuity_ave}, we have 
\begin{align*}
    \E\left[ \sup_{t\in [n,n+1)} |X(t)|^p\right] &\leq 2^{p} \left\{  \E\left[ |X(n)|^p\right] +\E\left[ \sup_{t\neq s\in [n,n+1]} |X(t)-X(s)|^p\right]\right\}\\
    &\leq C\exp \left( \left( \frac{\lambda p}{2} + \frac{\lambda'p}{2} + \bar{\gamma}(p) +o(1)\right) n \right),
\end{align*} 
where $\lambda' \in (\frac{2\lambda}{5}, \lambda)$ and $p>4$. 
Hence,   for any $\epsilon>0$, we have that as $n \rightarrow{\infty}$
\begin{equation}\label{borelsum}
    \begin{aligned}
    \P &\left\{ \sup_{t\in [n,n+1]} \frac{1}{|\Lambda_{L(t)}|} \int_{\Lambda_{L(t)}} (u(t,x)-1) dx   >\epsilon \right\}\\
    &\leq \frac{\E\left[ \sup_{t\in [n,n+1]} |X(t)|^p\right]}{|\Lambda_{L(n)}|^p\epsilon^p}\\
    &\leq C\epsilon^{-p} \exp \left(  \left( -\frac{\lambda p}{2} + \frac{\lambda'p}{2}+ \bar{\gamma}(p) +o(1)\right) n \right).
\end{aligned}
\end{equation}
As long as  $\lambda > \frac{5\bar{\gamma}(4)}{6}$, we can choose $\lambda'$ close to $\frac{2\lambda}{5}$ and $p$ close to $4$ so that $ -\frac{\lambda p}{2} + \frac{\lambda'p}{2}+ \bar{\gamma}(p)<0$. In other words, if $\lambda > \frac{5\bar{\gamma}(4)}{6}$, \eqref{borelsum} is summable over $n$, and the Borel-Cantelli lemma completes the proof.

\end{proof}

\section{Proof of Theorem \ref{clt}} \label{sec5}
In this section, we provide a proof of  Theorem \ref{clt} by following the strategy for the proof of the central limit theorem given in Cranston-Molchanov \cite{CM2}. We first estimate  the variance of  $\int_{\Lambda_{L}} u(t,x)dx$ as $t\to \infty$ where $\Lambda_L:=[-L, L]$ and $L:=L(t)$.  For the estimation of the variance of $\int_{-R}^R u(t,x)dx$ for a fixed time $t>0$ but as $R\to \infty$ is given in Proposition 3.1 in \cite{HNV}. From now on, $L$ always denote $L(t)$.

\begin{lemma}\label{varlower} Suppose $L(t)>0$ is an increasing function such that 
\[ \liminf_{t\to \infty} \frac{L(t)}{\sqrt{t}} = \infty. \]
Then we have 
\begin{equation}
\lim_{t\to\infty} \frac{\Var\left(\int_{\Lambda_L} u(t,x)dx \right)}{2L(t) \int_0^t \xi(r) \ dr} = 1, 
\end{equation}
    where $\xi(r):= \E\left[\sigma^2( u(r,0)) \right]$.
\end{lemma}
\begin{proof}
The proof is similar to Proposition 3.1 in \cite{HNV}. Due to the It$\hat{\text{o}}$ isometry, we have 
\begin{align*}
    \E\left[u(t,x)u(t,x') \right] &= 1+ \int_0^t \int_\R p_{t-r}(x-y)p_{t-r}(x'-y) \E\left[ \sigma^2(u(r,y))    \right]dy dr\\
    &= 1+ \int_0^t \int_\R\xi(r)p_{t-r}(x-y)p_{t-r}(x'-y) dy dr\\
    &= 1+ \int_0^t \xi(r) p_{2t-2r}(x-x')dr,
\end{align*}
where we have used the semigroup property in the last equality. Therefore, we have 
\begin{align*}
     \Var\left(\int_{\Lambda_L} u(t,x)dx \right) &= \int_{-L(t)}^{L(t)}\int_{-L(t)}^{L(t)}\int_0^t \xi(r) p_{2t-2r}(x-x')drdxdx' \\
     &= 2 \int_0^t \xi(r) \int_0^{2L(t)} p_{2t-2r}(z)(2L(t)-z)dzdr \\
     &= 2L(t)  \int_0^t \xi(r) \int_0^{2L(t)} p_{2t-2r}(z)\left(2-\frac{z}{L(t)}\right)dzdr.
\end{align*}
First of all, it is easy to see that 
\[ \sup_{0\leq r\leq t} \frac{1}{L(t)}\left| \int_0^{2L(t)} z \, p_{2t-2r} (z)\,  dz \right| \leq \frac{\sqrt{t}}{L(t)} \to 0 \quad \text{as $t\to \infty$}.\]
In addition, we also have
\begin{align*}
2\int_0^{2L(t)} p_{2t-2r}(z) \, dz = 1 - 2\int_{2L(t)}^\infty p_{2t-2r}(z) \, dz,
\end{align*}
and 
\[ \sup_{0\leq r\leq t}\left| \int_{2L(t)}^\infty p_{2t-2r}(z) \, dz\right| \leq \exp\left( - L^2(t)/t\right)\to 0 \quad \text{as $t\to \infty$}.\] 
Combining things together, we get the desired result. 
\end{proof}

The next lemma shows that it is enough to prove the central limit theorem for the time-dependent average of the localized random field $\left\{ u^{(k)}(t,x)\right\}$ which has independent property as long as $x$ and $y$ are far apart (see Lemma \ref{lem:localization}). 

\begin{lemma}\label{cltapprox} Let $L=L(t) = e^{\lambda t}$ for any $\lambda>0$.
For all $\epsilon>0$ and sufficiently large $k$, we have 
\begin{equation}\label{eq:cltapprox}
    \lim_{t\rightarrow{\infty}} \P\left\{\left| \frac{\int_{\Lambda_L} \left( u(t,x)-1  \right) dx  }{\sqrt{\Var{\left(\int_{\Lambda_L} u(t,x)dx \right)}}}- \frac{\int_{\Lambda_L} \left( u^{(k)}(t,x)-1  \right) dx  }{\sqrt{\Var{\left(\int_{\Lambda_L} u^{(k)}(t,x)dx \right)}}} \right| \geq\epsilon \right\} =0.
\end{equation}
\end{lemma}
\begin{proof}
We first show the following: For $k>  \left(\sqrt{ \bar{\gamma}(2)/2+2\lambda} \vee 2\right)$ and all large $t>0$, 
\begin{equation}\label{varapprox}
    \left|\Var{\left(\int_{\Lambda_L} u(t,x)dx \right)} - \Var{\left(\int_{\Lambda_L} u^{(k)}(t,x)dx \right)} \right| \leq e^{-c(k) t },
\end{equation}
where $c(k)>0$ is an increasing function of $k$ such that $\lim_{k\rightarrow{\infty}}c(k)= \infty$. To get this, we compute
\begin{equation}\label{eq:variance-difference}
\begin{aligned}
    &\left|\Var{\left(\int_{\Lambda_L} u(t,x)dx \right)} - \Var{\left(\int_{\Lambda_L} u^{(k)}(t,x)dx \right)} \right| \\
    &=\left| \E\left[\left(\int_{\Lambda_L} \left(u(t,x)-1\right)
    dx \right)^2-  \left(\int_{\Lambda_L} \left(u^{(k)}(t,x)-1\right)dx \right)^2   \right]\right|\\
    & = \left|\E\left[ \left(\int_{\Lambda_L} \left(u(t,x) - u^{(k)}(t,x) \right)dx \right)\left(\int_{\Lambda_L} \left(u(t,x)+u^{(k)}(t,x)-2 \right)dx \right)   \right]\right| \\
    &\leq  \E\left[ \left(\int_{\Lambda_L} \left(u(t,x) - u^{(k)}(t,x)\right) dx \right)^2  \right]^{\frac{1}{2}} \E\left[ \left(\int_{\Lambda_L} \left(u(t,x) + u^{(k)}(t,x)-2\right) dx \right)^2  \right]^{\frac{1}{2}}\\
    &\leq |\Lambda_L|^2\sup_{x\in\R} \E\left[|u(t,x)-u^{(k)}(t,x)|^2    \right]^{\frac{1}{2}}\sup_{x\in\R}\E\left[|u(t,x)+u^{(k)}(t,x)-2|^2    \right]^{\frac{1}{2}}\\
    &\leq C\exp \left\{ \left(2\lambda + \frac{\bar{\gamma}(2)}{2} - k^2 + o(1) \right)t      \right\},
\end{aligned}
\end{equation}
as $t\rightarrow{\infty}$. Here, we used the Cauchy-Schwarz inequality, Jensen's inequality and Lemma \ref{lem:localization} to get  the last inequality above. This proves \eqref{varapprox}. 

Now we define two terms 
\begin{align*}
    A_1&:= \left| \frac{\int_{\Lambda_L} \left( u(t,x)-1  \right) dx  }{\sqrt{\Var{\left(\int_{\Lambda_L} u(t,x)dx \right)}}}- \frac{\int_{\Lambda_L} \left( u^{(k)}(t,x)-1  \right) dx  }{\sqrt{\Var{\left(\int_{\Lambda_L} u(t,x)dx \right)}}} \right|,\\
    A_2&:= \left| \frac{\int_{\Lambda_L} \left( u^{(k)}(t,x)-1  \right) dx  }{\sqrt{\Var{\left(\int_{\Lambda_L} u(t,x)dx \right)}}}- \frac{\int_{\Lambda_L} \left( u^{(k)}(t,x)-1  \right) dx  }{\sqrt{\Var{\left(\int_{\Lambda_L} u^{(k)}(t,x)dx \right)}}} \right|.
\end{align*}
To show \eqref{eq:cltapprox}, it is enough to show $A_1,A_2 \rightarrow{0}$ in probability as $t\rightarrow{\infty}$. For $A_1$, we have  
\begin{align*}
    \E[A_1] & \leq \frac{|\Lambda_L| \sup_{x\in\R} \E\left[|u(t,x)-u^{(k)}(t,x)| \right]   }{\sqrt{\Var{\left(\int_{\Lambda_L} u(t,x)dx \right)}}}\\
    & \leq \frac{e^{\lambda t}\,  \sup_{x\in \R}  \E\left[|u(t,x)-u^{(k)}(t,x)| ^k\right]^{1/k}}{\sqrt{\Var{\left(\int_{\Lambda_L} u(t,x)dx \right)}}}.
\end{align*}
Thus, Lemmas \ref{lem:localization} and  \ref{varlower} implies that $\E[A_1]\rightarrow{0}$ as $t \rightarrow{\infty}$.

We now estimate $\E[A_2]$. By Jensen's inequality, we obtain that
\begin{align*}
    \E[A_2] &\leq \E\left[ \left| \int_{\Lambda_L} \left( u^{(k)}(t,x)-1  \right) dx\right| \right] \cdot \left| \frac{1}{\sqrt{\Var{\left(\int_{\Lambda_L} u(t,x)dx \right)}}}- \frac{1}{\sqrt{\Var{\left(\int_{\Lambda_L} u^{(k)}(t,x)dx \right)}}} \right| \\
    &\leq\sqrt{\Var{\left(\int_{\Lambda_L} u^{(k)}(t,x)dx \right)}} \cdot \left| \frac{1}{\sqrt{\Var{\left(\int_{\Lambda_L} u(t,x)dx \right)}}}- \frac{1}{\sqrt{\Var{\left(\int_{\Lambda_L} u^{(k)}(t,x)dx \right)}}} \right|.
\end{align*}
Moreover,
\begin{align*}
    &\left| \frac{1}{\sqrt{\Var{\left(\int_{\Lambda_L} u(t,x)dx \right)}}}- \frac{1}{\sqrt{\Var{\left(\int_{\Lambda_L} u^{(k)}(t,x)dx \right)}}} \right|\\
    &= \frac{\left|  \Var{\left(\int_{\Lambda_L} u^{(k)}(t,x)dx \right)}-\Var{\left(\int_{\Lambda_L} u(t,x)dx \right)}\right|}{\sqrt{\Var{\left(\int_{\Lambda_L} u^{(k)}(t,x)dx \right)}\Var{\left(\int_{\Lambda_L} u(t,x)dx \right)}} \left( \sqrt{\Var{\left(\int_{\Lambda_L} u^{(k)}(t,x)dx \right)}} +\sqrt{\Var{\left(\int_{\Lambda_L} u(t,x)dx \right)}}\right)}.
\end{align*}
Therefore,  we can deduce $\E[A_2]\rightarrow{0}$ as $t \rightarrow{\infty}$ by considering \eqref{varapprox} and Lemma \ref{varlower}.
\end{proof}

We now  prove Theorem \ref{clt}. For simplicity, we write $\Lambda_L:=[-L, L]$, $L:=e^{\lambda t}$ and $L':=e^{\lambda' t}$. 
\begin{proof}[Proof of Theorem \ref{clt}] We first show (i) of Theorem \ref{clt}.
By Lemma \ref{cltapprox}, it suffices to show that for sufficiently large $k$,
\begin{equation}
    \frac{\int_{\Lambda_L} \left( u^{(k)}(t,x)-1  \right) dx  }{\sqrt{\Var{\left(\int_{\Lambda_L} u^{(k)}(t,x)dx \right)}}} \xrightarrow{\mathcal{L}}N(0,1), \qquad \text{as }t\rightarrow{\infty}. 
\end{equation}
Using the decomposition of $\Lambda_{L}=S_L \cup \bigcup_{i\in \mathcal{I}}\Lambda'_i$ (see \eqref{partition2}), we have 
\begin{align}\label{eq:uk-decomp}
     \frac{\int_{\Lambda_L} \left( u^{(k)}(t,x)-1  \right) dx  }{\sqrt{\Var{\left(\int_{\Lambda_L} u^{(k)}(t,x)dx \right)}}}&=  \frac{\int_{S_L} \left( u^{(k)}(t,x)-1  \right) dx  }{\sqrt{\Var{\left(\int_{\Lambda_L} u^{(k)}(t,x)dx \right)}}}+
      \frac{\sum_{i\in \mathcal{I}}\int_{\Lambda'_i} \left( u^{(k)}(t,x)-1  \right) dx  }{\sqrt{\Var{\left(\int_{\Lambda_L} u^{(k)}(t,x)dx \right)}}}.
\end{align}
We claim that the first term on the RHS of \eqref{eq:uk-decomp} converges to $0$ in probability. To show this, we use a   partition $\{S_{L, i}\, : i=0, \dots, q:=\lfloor\frac{2L}{L'}\rfloor\}$ 
of $S_L$ (see \eqref{eq:partition_SL}). Since $\left\{ \int_{S_{L, i}} \left( u^{(k)}(t,x)-1  \right) dx \right\}_{i=0}^{q}$ is a collection of independent random variables by Lemma \ref{lem:localization}, we have 
\begin{align*}
\E\left[ \left|\int_{S_L} \left( u^{(k)}(t,x)-1  \right) dx \right|^2 \right] &=  \sum_{i=0}^{q} \E \left[\left|\int_{S_{L, i}} \left( u^{(k)}(t,x)-1  \right) dx \right|^2\right]\\
& \leq  4(q+1) \E \left[  \left|\int_{0}^{\lceil c_0t^2k^3 \rceil} \left( u^{(k)}(t,x)-1  \right) dx \right|^2\right], \end{align*}
where we used the stationarity in the last inequality. We now  partition $[-L, L]$ into subintervals of size $L'':=\lceil c_0t^2k^3 \rceil$ and call those subintervals $I_i$'s. Once again, thanks to Lemma \ref{lem:localization},  we can have a collection of i.i.d. random variables  $\left\{\int_{I_i} \left( u^{(k)}(t,x)-1  \right) dx \, : i=1, 3,  \dots \right\}$ so  that 
\begin{align*} 
\Var{\left(\int_{\Lambda_L} u^{(k)}(t,x)dx \right)} & \geq \sum_{i=1}^{\lfloor L/2L''\rfloor} \Var{\left(\int_{I_{2i-1}} u^{(k)}(t,x)dx \right)}\\
& = \bigg\lfloor\frac{L}{2L''}\bigg\rfloor \cdot \E \left[ \left|\int_0^{\lceil c_0t^2k^3 \rceil} \left( u^{(k)}(t,x)-1  \right) dx \right|^2\right].
\end{align*}
Since  $q\leq 2L/L'$ and $L''/L' \to 0$ as $t \to \infty$, we show that the first term on the right hand side of \eqref{eq:uk-decomp} converges to 0 in  $L^2(\Omega)$, which also implies the convergence in probability to 0. 

We now focus on the second term in \eqref{eq:uk-decomp}. Here, we apply the Lyapunov criterion for the central limit theorem. From the localization argument, we know that $\{ u^{(k)}(t,x) , x\in \Lambda'_i\}$ and $\{ u^{(k)}(t,x) , x\in \Lambda'_j\}$ are independent if $i\neq j$. Thus, we only need to  show that for any $\epsilon>0$, 
\begin{equation}\label{lyapunovcondi}
    \limsup_{t\rightarrow{\infty}} \frac{ \sum_{i\in \mathcal{I}} \E\left[\left| \int_{\Lambda'_i}(u^{(k)}(t,x) -1 ) dx \right|^{2+\epsilon} \right] } { \left( \sum_{i\in \mathcal{I}}  \Var{\left(\int_{\Lambda'_i} u^{(k)}(t,x) dx  \right)}\right)^{1+\epsilon/2}} = 0.
\end{equation} 
By Jensen's inequality, we obtain that 
\begin{align*}
    \sum_{i\in \mathcal{I}} \E\left[\left| \int_{\Lambda'_i}(u^{(k)}(t,x) -1 ) dx \right|^{2+\epsilon} \right] &\leq C\left( \frac{L}{L'} \right) \cdot (L')^{2+\epsilon} \exp \left\{ \left(\bar{\gamma}(2+\epsilon)  +o(1) \right)t\right\}\\
    &\leq  C \exp \left\{ \left((1+\epsilon)\lambda' + \lambda + \bar{\gamma}(2+\epsilon)+o(1) \right) t\right\},
\end{align*} as $t\rightarrow{\infty}$.
Once  we use  Lemma \ref{varlower} with \eqref{varapprox} (i.e. replace $L$ by $L'$ in Lemma \ref{varlower}), we have 
\begin{align*}
    \left( \sum_{i\in \mathcal{I}}  \Var{\left(\int_{\Lambda'_i} u^{(k)}(t,x) dx  \right)}\right)^{1+\epsilon/2} &\geq  C \left(\frac{L}{L'}\right)^{1+\epsilon/2} \left( 2L' \int_0^t \xi(r) dr \right)^{1+\epsilon/2}\\
    &\geq C L^{1+\epsilon/2} e^{\left(\barbelow{\gamma}(2)+o(1)\right)(1+\epsilon/2)t}\\
    &= C\exp \left\{ \left(  \lambda  +\barbelow{\gamma}(2)+o(1)    \right)\left(1+ \frac{\epsilon}{2}\right)t    \right\}.
\end{align*}
Therefore, we have  
\begin{multline}
    \frac{ \sum_{i\in \mathcal{I}} \E\left[\left| \int_{\Lambda'_i}(u^{(k)}(t,x) -1 ) dx \right|^{2+\epsilon} \right] } { \left( \sum_{i\in \mathcal{I}}  \Var{\left(\int_{\Lambda'_i} u^{(k)}(t,x) dx  \right)}\right)^{1+\epsilon/2}}\\
     \leq C \exp\left\{ \left((1+\epsilon)\lambda' - \epsilon \left( \frac{\lambda}{2} -\frac{\bar{\gamma}(2+\epsilon)-\barbelow{\gamma}(2)}{\epsilon}+\frac{\barbelow{\gamma}(2) }{2} \right)    +o(1)\right)t \right\}.
\end{multline}
Since  $\lambda'$ is arbitrary and $\lambda > \inf_{\epsilon>0} \left\{ \frac{2(\bar{\gamma}(2+\epsilon)-\barbelow{\gamma}(2))}{\epsilon}-\barbelow{\gamma}(2) \right\}$, we prove \eqref{lyapunovcondi}, which completes the proof of (i) in Theorem \ref{clt}. 

For the proof of (ii), as in the proof of (i), it suffices to prove that, for a sufficiently large $k$ and  for all $p \in (1,2]$, 
\begin{align}
E \left[\Bigg|   \frac{ \int_{\Lambda_L} (u^{(k)}(t,x)-1) dx }{ \sqrt{ \Var \left( \int_{\Lambda_L} u^{(k)}(t,x) dx \right) } }\Bigg|^p \right]  \rightarrow 0 \quad \text{as }t \rightarrow \infty,
\end{align} provided that $0< \lambda <  \frac{2p}{2-p} \left( \frac{\barbelow{\gamma}(2)}{2} - \frac{\bar{\gamma}(p)}{p} \right)$. Using \eqref{vonbahr} and Lemma \ref{varlower}, we have 
\begin{align}\label{eq:upperbd} 
\E \left[\left|   \frac{ \int_{\Lambda_L} [u^{(k)}(t,x)-1] dx }{ \sqrt{ \Var \left( \int_{\Lambda_L} u^{(k)}(t,x) dx \right) } }\right|^p \right] &\leq \exp\left\{ \left( \lambda+\lambda'(p-1) +\bar{\gamma}(p) - \frac{p}{2} \lambda - \frac{p}{2} \barbelow{\gamma}(2)+o(1)   \right)t \right\},
\end{align} for all $\lambda' \in (0, \lambda)$.  Since $\lambda'$ is arbitrary, as long as $0< \lambda <  \frac{2p}{2-p} \left( \frac{\barbelow{\gamma}(2)}{2} - \frac{\bar{\gamma}(p)}{p} \right)$, the right hand side of \eqref{eq:upperbd} converges to 0 as $t\to \infty$.  This  completes  the proof of (ii).

\end{proof}

\section{Proof of Theorem \ref{clt2}}\label{sec6}

In this section, we provide a proof of Theorem \ref{clt2} (the quantitative central limit theorem) by following the  proof of Theorem 1.1 in \cite{HNV} where the authors  showed the quantitative central limit theorem for the spatial average $\frac{1}{2R}\int_{-R}^R (u(t, x)-1) dx$ as $R \to \infty$ for fixed $t>0$ by using the Malliavin calculus and Stein's method (see also \cite[Corollary 2.7]{CKNP2} where the authors consider $N^{-d}\int_{[0,N]^d} (u(t_N, x)-1) dx$ where $t_N=o(\log N)$ as $N\to \infty$). Although our proof is very similar to the proof of Theorem 1.1 in \cite{HNV},  we prefer to give a self-contained proof for completeness and also our proof indicates  a precise length of the interval which guarantees the quantitative central limit theorem.  Note that the method combining the Malliavin calculus and Stein's method used in \cite{HNV, CKNP2} does not require any of techniques from the previous sections such as a localization argument to get independence. We first introduce briefly some basic notations and facts about the Malliavin calculus and Stein's method for the sake of completeness (see Section 2 of \cite{HNV} and also Chapter 5 by Nourdin and Peccati \cite{NP} for more details).  

Define $H:= L^2(\R_+ \times \R)$ and the Wiener integral $h\mapsto W(h)$ as 
\begin{equation*}
    W(h) = \int_0^\infty \int_{\R}h(t,x) W(dt\, dx) \quad \text{for $h \in H$}.
\end{equation*}
We denote by $D$ the Malliavin derivative operator and by $\delta$ the adjoint of the derivative operator $D$ given by the duality formula 
\begin{equation}
    \E[\delta(u)F] = \E[\left\langle u,DF \right\rangle_H]
\end{equation}
for any $F$ in the Gaussian Sobolev space $\D^{1,2}$ and any $u \in L^2( \Omega ;H)$ in the domain of $\delta$, denoted by $\text{Dom}\,\delta$. Note that  any predictable random field $X$ which satisfies 
\begin{equation}
    \int_0^\infty \int_\R \|X(s,y)\|_2^2 \, dyds < \infty
\end{equation} belongs to $\text{Dom}\,\delta$ and $\delta(X)$ coincides with the Walsh integral. Here, we define  $\|X\|_k:=\{\E[ |X|^k]\}^{1/k}$ for a random variable $X$.  Thus,  we can write the solution $u(t,x)$ to \eqref{SHE} as 
\begin{equation}
    u(t, x) = 1+ \delta( p_{t-\cdot}(x-*)u(\cdot,*)).
\end{equation}
In addition, it is well-known  that $u(t, x)$ belongs to $\D^{1,p}$ for any $p\geq 2 $  for any $(t,x)$  and the Malliavin  derivative of $u$ satisfies the following  equation for a.e. $(s,y)\in (0,t)\times \R$ 
\begin{align*}
    D_{s,y} u(t,x) = &p_{t-s}(x-y) \sigma (u(s,y)) \\
    &+ \int_s^t \int_\R p_{t-r}(x-z) \Sigma(r,z) D_{s,y}u(r,z) W(drdz), \numberthis \label{solderivative}
\end{align*}
where $\Sigma(r,z)$ is an adapted process, bounded by $\lip_\sigma$. If $\sigma$ is continuously differentiable, then $\Sigma(r,z) = \sigma'(u(r,z))$. From Lemma 4.2 in \cite{CKNP}, we know that for all  $0<\epsilon<1, t \geq s\geq r >0,$ $k\geq 2 $, and for all $x \in \R$, 
\begin{equation} \label{derivestimate}
    \lVert D_{r,z}u(s,y) \rVert_k \leq  C_1(\sigma, \epsilon) \, p_{s-r}(y-z)  \exp\left\{ C_2(\sigma, \epsilon, k) t \right\} 
\end{equation}
holds for a.e. $(r,z) \in (0,s) \times \R$ where 
\begin{align}\label{constderivative}
    C_1(\sigma, \epsilon)&:= \frac{8\lip_\sigma }{\epsilon^{3/2}} \quad \text{and}\quad
    C_2(\sigma, \epsilon, k):=  \frac{2^3 k^2 (\lip_\sigma)^4}{(1-\epsilon)^4}.
\end{align}
Note that  $C_2(\sigma, \epsilon, k)$  in \eqref{constderivative} was obtained from the following definition in  Lemma 4.2 of \cite{CKNP}:
\[ C_2(\sigma, \epsilon, k):=2\zeta\left( \frac{a(\epsilon)}{k}\right),\]
where, for all $x\geq 0$, 
\begin{align*}
    a(\epsilon)&:= \frac{(1-\epsilon)^2}{2^{3/2}(\lip_\sigma)^2 }, \quad  \zeta(x):=\rho^{-1}(x) \quad  \text{and} \quad \rho(x):= \frac{1}{\pi}\int_{\R} \frac{1}{2x+|\xi|^2}d\xi = \frac{1}{\sqrt{2x}}.
\end{align*}
We now define the total variance distance between two random variables $F$ and $G$ by 
\begin{equation}
    d_{TV}(F,G)= \sup_{B \in \mathcal{B}(\mathbb{R)}} | \P[F\in B] - \P[G\in B]|,
\end{equation}where $\mathcal{B}(\R)$ is the Borel sigma field on $\R$. The next proposition which combines the Malliavin calculus and Stein's method is given in \cite{HNV} (see also  Nourdin-Peccati \cite{NP}). 

\begin{proposition}[Proposition 2.2 in \cite{HNV}]
Let $F \in \D^{1,2}$ be given by  $F=\delta(u)$ for some $u \in \textup{Dom}\ \delta $. Let $Z \sim N(0,1)$.  If  $\E[F^2] =1$, then we have 
\[ d_{TV}(F, Z) \leq 2\sqrt{\Var\langle DF, u \rangle_H}.\] 
\end{proposition}

We now provide a proof of Theorem \ref{clt2} by following the proof of Theorem 1.1 of \cite{HNV}. 
\begin{proof}[Proof of Theorem \ref{clt2}] Let $L:=L(t):=e^{\lambda t}$ for $\lambda >0$. Define  
\[\sigma_\lambda(t) := \sqrt{\Var{\left(\int_{\Lambda_L} u(t,x)dx \right)}} \quad \text{and}\quad  F_\lambda(t) : = \frac{1}{\sigma_\lambda(t)} \left( \int_{\Lambda_L} [ u(t,x) -1 ]dx\right).\]
 By stochastic Fubini's theorem and \eqref{solution}, we can write  
\begin{align*}
    F_\lambda(t) = \int_0^t \int_\R \left( \frac{1}{\sigma_\lambda(t)} \int_{\Lambda_L} p_{t-s}(x-y)\sigma (u(s,y))dx\right) W(ds\, dy).
\end{align*} Thus, for any fixed $t\geq 0$, $F_\lambda(t)= \delta(v_\lambda)$ where 
\begin{equation}
    v_\lambda(s,y) := v_\lambda^{(t)}(s,y):= 1_{[0,t]}(s) \frac{1}{\sigma_\lambda(t)} \int_{\Lambda_L}  p_{t-s}(x-y)\sigma (u(s,y))dx.
\end{equation}
Since  the derivative operator $D$ can be interchanged with Lebesgue integration, we have 
\begin{equation}
    D_{s,y}F_\lambda(t)= 1_{[0,t]}\frac{1}{\sigma_\lambda(t)}\int_{\Lambda_L}  D_{s,y} u(t,x)dx.
\end{equation}
Therefore, we get
\begin{align*}
    \left\langle DF_\lambda(t), v_\lambda \right\rangle_H = \frac{1}{\sigma_\lambda^2(t)} \int_0^t \int_\R \int_{\Lambda_L}\int_{\Lambda_L} p_{t-s}(x-y) \sigma (u(s,y)) D_{s,y} u(t,x') dxdx'dyds.
\end{align*} Our goal is to find $\lambda$ which guarantees  $\Var \left( \left\langle DF_\lambda(t), v_\lambda \right\rangle_H \right) \to 0$ as $t\to \infty$. 
From \eqref{solderivative}, we can compute 
\begin{align*}
 \left\langle DF_\lambda(t), v_\lambda \right\rangle_H =& \frac{1}{\sigma_\lambda^2(t)} \int_0^t \int_\R \left(\int_{\Lambda_L} p_{t-s}(x-y)dx \right)^2 \sigma^2(u(s,y))dyds \\
 &+ \frac{1}{\sigma_\lambda^2(t)} \int_0^t \int_\R \int_{\Lambda_L}\int_{\Lambda_L}  p_{t-s}(x-y) \sigma(u(s,y)) \\
 &\qquad\qquad\times \left(\int_s^t \int_\R p_{t-r}(\Tilde{x}-z) \Sigma (r,z) D_{s,y}u(r,z) W(drdz)  \right) dxd\Tilde{x}dyds.
\end{align*}
Using Minkowski's inequality for integrals, we can see that 
\begin{equation*}
    \sqrt{\Var\left( \left\langle DF_\lambda(t), v_\lambda \right\rangle_H \right)} \leq B_1+B_2,
\end{equation*}
where 
\begin{align*}
    B_1:= &\frac{1}{\sigma_\lambda^2(t)}  \int_0^t \left( \int_{\R^2}    \left(\int_{\Lambda_L} p_{t-s}(x-y) dx \right)^2 \left( \int_{\Lambda_L} p_{t-s}(x'-y-) dx'\right)^2 \right. 
    \\ & \times \Cov{\left( \sigma^2(u(s,y)), \sigma^2 (u(s,y'))\right)dydy'}\bigg)^{\frac{1}{2}} ds 
 \end{align*}
and
\begin{align*}
    B_2:= &\frac{1}{\sigma_\lambda^2(t)} \int_0^t \left( \int_{\R^2} \int_{(\Lambda_L)^4} p_{t-s}(x-y)p_{t-s}(x'-y') \int_s^t \int_\R p_{t-r}(\Tilde{x}-z)p_{t-r}(\Tilde{x}'-z) \right.\\
    &\times \E \left[\sigma(u(s,y))\sigma(u(s,y'))\Sigma^2(r,z) D_{s,y}u(r,z)D_{s,y'}u(r,z) \right] dzdrdxdx'd\Tilde{x}d\Tilde{x}'dydy' \bigg)^\frac{1}{2} ds.
\end{align*}
We first estimate $B_1$. Applying the  Clark-Ocone formula to $\sigma^2(u(s,y))$ as in \cite{HNV}, we have 
\begin{align*}
         \Cov{\left( \sigma^2(u(s,y)), \sigma^2 (u(s,y'))\right)} = \int_0^s \int_\R \E \left[  \E \left[ D_{r,z}(\sigma^2(u(s,y)))|\mathcal{F}_r \right] \E \left[ D_{r,z}(\sigma^2(u(s,y')))|\mathcal{F}_r \right]\right]dzdr.
\end{align*} 
In addition, as it was shown in \cite{HNV}, we have 
\begin{equation*}
    \left\lVert \E \left[ D_{r,z}(\sigma^2(u(s,y)))|\mathcal{F}_r \right]  \right\rVert_2 \leq CK_4(t) \lVert D_{r,z}u(s,y) \rVert_4,
\end{equation*}
where $C$ is a positive constant which depends only on $\lip_\sigma$ and $K_p(t)$ is defined as 
\begin{equation*}
    K_p(t) = \sup_{0\leq s \leq t}\sup_{y\in \R}\lVert \sigma(u(s,y)) \rVert _p \qquad \text{for $p\geq 2$}.
\end{equation*}
Since $|\sigma(u)|\leq  \lip_\sigma |u|$, we have 
\begin{equation}
    K_p(t)\leq  \lip_\sigma e^{\left(\frac{\bar{\gamma}(p)}{p} + o(1) \right)t}, 
\end{equation}
as $t\rightarrow{\infty}$. Let $0<\epsilon<1$ be fixed and $C_2=C_2( \sigma, \epsilon, 4)$ in \eqref{constderivative}. We denote by $C$ a constant that depends on $\lip_\sigma$ and can be changed line by line. Using \eqref{derivestimate}, we have
\begin{align*}
    \lvert  \Cov{\left( \sigma^2(u(s,y)), \sigma^2 (u(s,y'))\right)} \rvert &\leq C K^2_4(t) \int_0^s \int_\R \lVert D_{r,z}u(s,y) \rVert_4\lVert D_{r,z}u(s,y') \rVert_4 dzdr \\
    &\leq C e^{\left(2C_2 + \frac{\bar{\gamma}(4)}{2}+o(1) \right)t} \int_0^s \int_\R p_{s-r}(y-z) p_{s-r}(y'-z) dzdr\\
    &= Ce^{\left(2C_2 + \frac{\bar{\gamma}(4)}{2}+o(1) \right)t}\int_0^s p_{2s-2r}(y-y')dr.
\end{align*} Moreover,  Lemma \ref{varlower} implies that as $t\to \infty$
\begin{align*}
    \sigma_\lambda^2(t)\geq L(t)\int_0^t \xi(r) dr \geq  e^{\left(\lambda +\barbelow\gamma(2)+o(1)\right)t}. 
\end{align*}  We  may also follow  \cite{HNV} (i.e. replace $R$ by $L=e^{\lambda t}$)  to see that 
\begin{align*} 
&\int_0^t  \left(  \int_{\R^2}    \left(\int_{\Lambda_L} p_{t-s}(x-y) dx \right)^2 \left( \int_{\Lambda_L} p_{t-s}(x'-y-) dx'\right)^2   \int_0^s p_{2s-2r}(y-y')drdydy' \right)^{\frac{1}{2}} ds \\
&\leq C \int_0^t \left( \int_0^s \int_{(\Lambda_L)^2} p_{2t-2r}(x-x') dxdx'dr \right)^{\frac{1}{2}} ds \\
&\leq C t^{3/2} |\Lambda_L|^{1/2}.
\end{align*}
Combining things together, we have 
\begin{align*}
     B_1\leq   C  \exp \left\{  -\left(\frac{1}{2}\lambda+\barbelow{\gamma}(2)-2C_2 - \frac{\bar{\gamma}(4)}{2} +o(1) \right)t   \right\}
\end{align*}
as $t\to \infty$. 

Now we estimate $B_2$. By H$\ddot{\text{o}}$lder's inequality, 
\begin{align*}
    \E &\left[\sigma(u(s,y))\sigma(u(s,y')) \Sigma^2(r,z) D_{s,y}u(r,z)D_{s,z'}u(r,z) \right] \\ 
    &\leq C K_4^2(t) \lVert D_{s,y}u(r,z)\rVert_4 \lVert D_{s,y'}u(r,z)\rVert_4.
\end{align*}
Again using \eqref{derivestimate}, we have 
\begin{align*}
    B_2 \leq C\frac{e^{\left(2C_2 + \frac{\bar{\gamma}(4)}{2}+o(1) \right)t}}{ e^{( \lambda + \barbelow{\gamma}(2)+o(1))t} }& \int_0^t \left( \int_{\R^2}\int_{(\Lambda_L)^4} \int_s^t \int_\R p_{t-s}(x-y)p_{t-s}(x'-y') p_{t-r}(\Tilde{x}-z) \right. \\ 
     &\times p_{t-r}(\Tilde{x}'-z)p_{r-s}(z-y) p_{r-s}(z-y')dzdrdxdx'd\Tilde{x}d\Tilde{x}'dydy'  \bigg)^{\frac{1}{2}}ds.
\end{align*}
Following \cite{HNV}, i.e., replacing $R$ by $L=e^{\lambda t}$, we have 
\begin{align*}
& \int_0^t \left( \int_{\R^2}\int_{(\Lambda_L)^4} \int_s^t \int_\R p_{t-s}(x-y)p_{t-s}(x'-y') p_{t-r}(\Tilde{x}-z) p_{t-r}(\Tilde{x}'-z)p_{r-s}(z-y)\right. \\ 
     &\qquad\qquad\times  p_{r-s}(z-y')dzdrdxdx'd\Tilde{x}d\Tilde{x}'dydy'  \bigg)^{\frac{1}{2}}ds\\
     &\leq \int_0^t \left( \int_{(\Lambda_L)^2} \int_s^t p_{2t+2r-4s}(x-x') drdxdx' \right)^{\frac{1}{2}}ds\\
     &\leq C t^{3/2}|\Lambda_L|^{1/2}. 
\end{align*}
Hence, we obtain
\begin{align*}
   B_2\leq  C \exp \left\{ -\left(\frac{1}{2}\lambda  +\barbelow{\gamma}(2) -2C_2 - \frac{\bar{\gamma}(4)}{2} +o(1) \right)t\right\} 
\end{align*} as $t\rightarrow{\infty}$.
We note that $C_2=C_2(\sigma, \epsilon,4)= 2^{7}\lip_\sigma^4/(1-\epsilon)^4$ strictly decreases as $\epsilon \rightarrow{0}$. Therefore, if we redefine $C_2 = 2^{7}\lip_\sigma^4 $ and choose $\lambda$ so that 
\begin{equation}\label{conditionlambda}
    \lambda >2^{9}\lip_\sigma^4 + \bar{\gamma}(4) - 2\barbelow{\gamma}(2),
\end{equation}
then this completes the proof by taking $\epsilon$ close to $0$.
\end{proof}
\begin{remark}
It is known that in the case of $\sigma(u)=\lip_\sigma u$, $\bar{\gamma}(2) = \gamma(2)= \lip_\sigma^4/4$ (see Corollary 2.8 in \cite{FD09}). Thus, the condition on $\lambda$ in \eqref{conditionlambda} becomes 
\begin{equation}
    \lambda >(2^{11}-2)\gamma(2)+ \gamma(4).
\end{equation}
Since we know that $\gamma(p) = \lip_\sigma^4 p(p^2-1)/24$ for all $p>0$ (see \cite{GY}), we can see that $(2^{11}-2)\gamma(2)+ \gamma(4) > 2\gamma'(2)-\gamma(2)$. For the general $\sigma$, one may use the moment comparison principle (\cite{JKM, CK2}) to get the bound on $\lambda$. 
\end{remark}

\section*{Acknowledgments} We appreciate  Davar Khoshnevisan  for stimulating discussions and suggestions. We also thank Carl Mueller and Mohammud Foondun for useful discussions. 

\begin{spacing}{1}
\begin{small}
\end{small}\end{spacing}
\vskip.1in

\begin{small}
\noindent\textbf{Kunwoo Kim} [\texttt{kunwookim@postech.ac.kr}]\\
\noindent Pohang University of Science and Technology (POSTECH), Pohang, Gyeongbuk, South Korea \\

\noindent\textbf{Jaeyun Yi} [\texttt{stork@postech.ac.kr}]\\
\noindent Pohang University of Science and Technology (POSTECH), Pohang, Gyeongbuk, South Korea

\end{small}

\end{document}